\documentclass[11pt]{amsart}
\usepackage[active]{srcltx}
\usepackage{amsmath, amsfonts,amsthm,times,graphics,color}
 \makeatletter
\renewcommand*\subjclass[2][2000]{%
  \def\@subjclass{#2}%
  \@ifundefined{subjclassname@#1}{%
    \ClassWarning{\@classname}{Unknown edition (#1) of Mathematics
      Subject Classification; using '1991'.}%
  }{%
    \@xp\let\@xp\subjclassname\csname subjclassname@#1\endcsname
  }%
}
 \makeatother
\usepackage{enumerate,amssymb,  mathrsfs,yhmath}

\newtheorem{theorem}{Theorem}[section]
\newtheorem{lemma}[theorem]{Lemma}
\newtheorem*{lemma*}{Lemma}

\newtheorem{corollary}[theorem]{Corollary}

\newtheorem{problem}[theorem]{Problem}

\def\1ton{1,2,\ldots,n}

\def\ID{{\Bbb D}}

\usepackage{amssymb}
\usepackage{amsthm}
\usepackage{mathrsfs, amsfonts, amsmath}
\usepackage{graphicx}

\newcommand{\R}{\mathbb{R}}

\newcommand{\csch}{\operatorname{csch}}
\newcommand{\dif}{\text{dif}}

\theoremstyle{definition}

\newtheorem{example}[theorem]{Example}

\theoremstyle{remark}
\newtheorem{remark}[theorem]{Remark}

\numberwithin{equation}{section}

\renewcommand{\imath}{i} 

\def\XXint#1#2#3{{\setbox0=\hbox{$#1{#2#3}{\int}$}
\vcenter{\hbox{$#2#3$}}\kern-.5\wd0}}

\def\ge{\geqslant}
\begin{document}
\title[Schwarz lemma for real functions onto surfaces ]{Schwarz lemma for real harmonic functions onto surfaces with non-negative Gaussian curvature} \subjclass{Primary 30F15  }


\author{David Kalaj}
\address{University of Montenegro, Faculty of Natural Sciences and
Mathematics, Cetinjski put b.b. 81000 Podgorica, Montenegro}
\email{davidk@ucg.ac.me}

\author{Miodrag Mateljevi\'c}
\address{Faculty of mathematics, University of Belgrade, Studentski Trg 16,
Belgrade, Republic of Serbia}
\email{miodrag@matf.bg.ac.rs}

\author{Iosif Pinelis}
\address{Department of Mathematical Sciences
Michigan Technological University
1400 Townsend Drive
Houghton, Michigan 49931-1295
U.S.A.}
\email{ipinelis@mtu.edu}

\begin{abstract}
Assume that $f$ is a real $\rho$-harmonic function of the unit disk $\mathbb{D}$ onto the interval $(-1,1)$, where $\rho(u,v)=R(u)$ is a metric defined in the infinite strip $(-1,1)\times \mathbb{R}$. Then we prove that $|\nabla f(z)|(1-|z|^2)\le \frac{4}{\pi}(1-f(z)^2)$ for all  $z\in\mathbb{D}$, provided that $\rho$ has a non-negative Gaussian curvature. This extends several results in the field and answers to  a conjecture proposed by the first author in 2014. Such an inequality is not true for negatively curved metrics.

\end{abstract}
\maketitle
\tableofcontents
\section{Introduction}
\subsection{Schwarz lemma}
The standard Schwarz lemma states that if $f$ is a holomorphic mapping of the unit disk $\mathbb{D}$ into itself such that $f(0)=0$ then $|f(z)|\le |z|$.

Its counterpart for harmonic mappings states the following (\cite[Section~4.6]{Duren2004}).  Let $f$ be a complex-valued function harmonic in the unit disk $\mathbb{D}$ into itself, with $f (0) = 0$. Then $$| f (z)| \le \frac{ 4}{\pi}\tan^{-1} |z|,$$ and this inequality is sharp for
each point $z \in \mathbb{D}$. Furthermore, the bound is sharp everywhere (but is attained
only at the origin) for univalent harmonic mappings $f$ of $\mathbb{D}$ onto itself with
$f (0) = 0.$

The standard Schwarz lemma (also called Schwarz--Pick lemma) for holomorphic mappings states that every holomorphic mapping $f$ of the unit disk onto itself satisfies the inequality \begin{equation}\label{schar}
|f'(z)|\le \frac{1-|f(z)|^2}{1-|z|^2}.
\end{equation}
A very important version of the Schwarz lemma for holomorphic functions has been obtained by Ahlfors \cite{A}, who proved the following: Let $f$ be a holomorphic map of the unit
disk $\mathbb{D}$ into a Riemann surface $S$ endowed with a Riemannian metric $\rho$ with Gaussian
curvature $\mathcal{K} \le -1$. Then the hyperbolic length of any curve in $\mathbb{D}$ is no less than
the length of its image. Equivalently,
$$ d_\rho(f(z), f(w)) \le  d_h(z, w)\text{ for all $z, w \in  \mathbb{D}$}$$
or
$ \|df(z)\|\le  1$ everywhere,
where the norms are taken with respect to the hyperbolic metric on $\mathbb{D}$ and the given
metric on the image. For some other generalizations of the Schwarz lemma we refer to the papers of  Yau \cite{Y, Ya2}, Osserman \cite{oser},  Yang and Zheng \cite{YaZh}, Royden \cite{royden}, Ni,
\cite{Ni1, Ni2} and Broder \cite{Br1}.
The recent survey by Broder \cite{Br2} also provides references to the Schwarz lemma in other
contexts. Most of mentioned papers deal with Schwarz lemma for holomorphic functions, and the target space has a non-positive curvature.

We refer as well to some generalizations of Schwarz lemma for harmonic functions the papers \cite{kavu,cvee5, colo, bur, mmarkovic, mat1, mat2, petar, apde}.

In particular, the following result was proved in \cite{kavu}.
If $f\colon\mathbb{D}\to (-1,1)$ is a real harmonic mapping, then it satisfies the inequality $$|\nabla f(z)|\le \frac{4}{\pi}\frac{1-f(z)^2}{1-|z|^2}\text{ for all }  z\in\mathbb{D}.$$ Later, by using the same  approach as that in \cite{kavu},
Chen \cite{chenm} improved the latter inequality by showing that
$$|\nabla f(z)|\le \frac{\cos(\frac{\pi}{2}f(z))}{1-|z|^2}\text{ for all } z\in\mathbb{D}.$$

In order to state our main result, let us introduce the class of $\varrho$-harmonic mappings.

\subsection{$\varrho$-Harmonic mappings}
Assume that $\Omega$ is a connected open set in the complex plane. Assume that  $\varrho$ is a positive continuous function in $\Omega$.
Then (by abusing the notation) it defines a conformal metric $\varrho(z)= \varrho(z)dz\otimes \overline{dz}$ in $\Omega$.
Then $\varrho$ defines a  Riemann surface $(\Omega, \varrho)$.
Moreover, assume that $\varrho$ is a smooth function in $\Omega$ with 
Gaussian
curvature $\mathcal{K}_\varrho$, 
where
\begin{equation}\label{gaus}\mathcal{K}_\varrho(z):=-\frac{\Delta \log \varrho(z)}{\varrho^2(z)}.\end{equation}
Here $\Delta$ denote the usual Laplacian: $$\Delta g(z):=g_{xx}+g_{yy},\ z=x+\imath y.  $$ We assume $\sup_{z\in \Omega} |\mathcal{K}_\varrho(z)|<\infty$ and $\varrho$ has a finite area defined by
$$\mathcal{A}(\varrho)=\int_{\Omega}\varrho^2(u+iv)\, du\, dv. 
$$

Let $f : (D, \delta) \to  (\Omega, \varrho)$  be a $\mathcal{C}^2$ map two  Riemann surfaces,
where $\delta$ is the (pullback to D via the inclusion of the) Euclidean metric. We say that
$f$ is harmonic if
\begin{equation}\label{el}
f_{z\overline z}+({(\log \varrho^2)}_w\circ f) \cdot f_z\,f_{\bar z}=0,
\end{equation}
where $z$ and $w$ are holomorphic coordinates on $D$ and $\Omega$,
respectively.
Recall that a Euclidean harmonic function $f$ is a solution of the Laplace equation $\Delta f=0$ and in this case $\varrho\equiv 1$.
 Also, $f$ satisfies \eqref{el} if and only if its Hopf
differential 
\begin{equation}\label{anal}
\mathrm{Hopf}(f):=(\varrho^2 \circ f) f_z\overline{f_{\bar z}}
\end{equation}
is a
holomorphic quadratic differential on $D$.

Assume $f : \mathbb{D} \to (-1, 1)$ is a real $\rho$-harmonic function. Vuorinen and the first
named author in \cite{kavu} introduced the quantity
\begin{equation}\label{sf}S(f):=|\nabla f(z)| \frac{1-|z|^2}{1-|f(z)|^2}\end{equation}
and showed that $S(f) \le \frac{4}{\pi}$
for Euclidean harmonic functions.
In order to extend the results in \cite{kavu}, the first named author in \cite{rocky} defined the class of admissible metrics. We say that a metric $\varrho$ is admissible if $\varrho(z)=\varphi(|z|)$, where $\varphi:  \mathbb{D}  \to \mathbb{C}\setminus (-\infty,0]$
 is an analytic function defined in the
unit disk satisfying the following properties:
\begin{enumerate}
\item   $\varphi(|z|) \le |\varphi(z)|$  and $\varphi$ is nonincreasing in $[0, 1]$,
\item  $\varphi(-1,1)\subset \mathbb{R}$  and $\int_0^1(\sqrt{\varphi(x)}-\sqrt{\varphi(-x)})dx=0$.
\end{enumerate}\
Then inequality \eqref{sf} was extended
by the first named author \cite{rocky} to $\varrho$-harmonic functions, where $\varrho$ is an \emph{admissible metric}.
The following question was posed in \cite{rocky}
\begin{problem}\label{prob} Let $f : \mathbb{D}\to  (-1, 1)$ be a real $\varrho$-harmonic function. Suppose $\varrho$
has non-negative Gaussian  curvature. Does the bound S$(f) \le\frac{4}{\pi}$
hold?
\end{problem}

\begin{remark}
{\it The assumption that the target domain has a non-negative Gaussian curvature is a crucial, and this is shown in Example~\ref{myex}.
This problem is somehow complementary to the already mentioned Scharz lemma type result of Ahlfors for holomorphic functions of the unit disk onto a surface with a non-positive Gaussian curvature.}
\end{remark}
We will see in the Example~\ref{primjer} below that the answer to the question posed in Problem~\ref{prob} is no. However,
it will be shown in this paper
that a real $\varrho$-harmonic function is also harmonic with respect to the \emph{modified} metric $\rho(u,v)=\varrho(u,0)$, and the positiveness of the Gaussian curvature of $\rho$ 
will be crucial.

Indeed, we shall prove the following theorem, which is the main content of this paper.
\begin{theorem}\label{prop}
Assume that $f$ is a real $\varrho$-harmonic function of the unit disk onto the interval $(-1,1)$. If $\rho(u,v)=\varrho(u,0)$, then $f$ is $\rho$-harmonic. Assume further that the Gaussian curvature of $\rho$ is non-negative. Then we have the sharp  inequality \begin{equation}\label{sf1}S(f)=|\nabla f(z)| \frac{1-|z|^2}{1-|f(z)|^2}\le \frac{4}{\pi} \text{\ \ for all }   z\in\mathbb{D}.\end{equation}
\end{theorem}
\begin{corollary}\label{rrje}
Assume that $\Omega$ is a hyperbolic domain in the complex plane and let $\lambda=\lambda_\Omega$ be its hyperbolic metric of constant Gaussian curvature equal to $-4$. Let $f:\Omega\to (-1,1)$ be a $\rho$-harmonic function, where $\rho(u,v)=R(u)$ has a non-negative Gaussian curvature. Then we have the following sharp inequality: \begin{equation}
d_h(f(z),f(w))\le \frac{4}{\pi}d_\lambda(z,w)\text{\ \;for all } z,w\in\Omega.
\end{equation}
Here $d_h$ is the hyperbolic metric in the unit disk defined by $$d_h(z,w) =\tanh^{-1}\frac{|z-w|}{|1-z\bar w|}.$$
\end{corollary}
The proof of the first part of Theorem~\ref{prop} is an easy matter and it is presented in Section~\ref{subse1}, while the second part is the content of Theorem~\ref{main}.
Corollary~\ref{rrje} is a straightforward application of the definition of the hyperbolic metric. We only need to notice the following. If $g\colon \mathbb{D}\to \Omega$ is a covering map, then $h(z) = f(g(z))$ is a real $\rho$-harmonic mapping of the unit disk onto $(-1,1)$. Moreover,  $$\lambda_\Omega(g(z))=\lambda_{\mathbb{D}}(z) |g'(z)|.$$  So, for $w=g(z)$, in view of \eqref{sf1} we have $$\frac{|\nabla f(w)|}{\lambda_\Omega(w)}={|\nabla h(z)|}(1-|z|^2)\le \frac{4}{\pi}(1-h(z)^2)=\frac{4}{\pi}{(1-|f(w)|^2)}.$$ Thus, $$\frac{|\nabla f(w)|}{(1-|f(w)|^2)}\le\frac{4}{\pi}\lambda_\Omega(w).$$

By integrating the previous inequality throughout the family of paths joining $z_1$ and $z_2$ (as at the end of the proof of Theorem~\ref{main}), we get
 \begin{equation}
d_h(f(z_1),f(z_2))\le \frac{4}{\pi}d_\lambda(z_1,z_2) \text{ for all } z_1,z_2\in\Omega.
\end{equation}

\subsection{Real $\varrho$-harmonic mappings and our setting (Real  $R-$harmonic mappings)}\label{subse1}
If $f$ is real, then \eqref{el} can be re-stated as follows:
\begin{equation}\label{ele}\Delta f +\frac{\varrho_u(f(z),0)-\imath \varrho_v(f(z),0)}{\varrho(f(z),0)}(f_x^2+f_y^2)=0.\end{equation} In particular, we see that  $\varrho_v(u,0)\equiv 0$ or $f$ is a constant function.

Let $R(u) = \varrho(u,0)$. If $f$ is a real harmonic function  of the unit disk onto the interval $(\alpha,\beta)$, then \begin{equation}\label{8may}\Delta f +\frac{R'(f)}{R(f)}(|\nabla f|^2) =0,\end{equation} where $R$ is a metric defined in the interval $(\alpha,\beta)$. Observe that $R$ can be extended to the infinite strip-domain $S(\alpha, \beta):=\{x+\imath y, x\in (\alpha,\beta), y\in\mathbb{R}\}$ by setting $\rho(u,v)=R(u)=\varrho(u,0)$.

Moreover, we have this important fact: \emph{$f$ is real $\varrho$-harmonic if and only if $f$ is real $\rho$-harmonic.} This is why we will consider the Gaussian curvature of $\rho$ instead of $\varrho$. We will refer to such real harmonic mappings as \emph{real  $R-$harmonic mappings}.

The Gaussian curvature of $\rho$ is given by
\begin{equation}\label{gkrr}\mathcal{K}_\rho(u,v) = -\frac{1}{R(u)^2}\left(\frac{R'(u)}{R(u)}\right)'.\end{equation}
In fact, equation \eqref{8may} is equivalent to the Laplace equation  $$\Delta g = 0,$$ where \begin{equation}\label{gg}g:=\frac{H(f)}{H(1)}:\ID\to (-1,1),\end{equation} while
\begin{equation}\label{nzero}H(u):=-\frac{1}{2}\left(\int_0^1 R(u)du+\int_0^{-1}R(u)du\right)+\int_0^u R(u)du,\end{equation} and $$H(1) = \frac{1}{2}\int_{-1}^1 R(u) du<\infty,$$ provided that $R$ belongs to the Lebesgue space $\mathcal{L}^1(-1,1)$. 
We will, however, prove that this condition $R\in \mathcal{L}^1(-1,1)$ is a priori satisfied for metrics of non-negative Gaussian curvature, with which we deal in our main result.

The following theorem contains some results for metrics which are not necessarily positively curved.

\begin{theorem}\label{kalajpos}
Assume that $f$ is a real $R$-harmonic mapping of the unit disk into the interval $(-1,1)$, and assume that $R$ is an increasing function in $(-1,0)$ and decreasing in $(0,1)$.
Then we have the following sharp inequality
\begin{equation}\label{222}|\nabla f(z)|\le 2\frac{1-|f(z)|}{1-|z|^2} \text{\ \,for all }  z\in\mathbb{D}.\end{equation} If $f(0)=0$ and $\int_{-1}^0R(t)dt =\int_0^1R(t)dt$, then we have
the sharp inequality
\begin{equation}\label{555}|f(z)|\le \frac{4}{\pi}\tan^{-1} |z|, \ z\in\mathbb{D}.\end{equation}
\end{theorem}

The proof of Theorem~\ref{kalajpos} is presented in Subsection~\ref{jun2022}. We also have the following straightforward corollary of Theorem~\ref{kalajpos}.
\begin{corollary}
If $R$ is even in $(-1,1)$ and decreasing in $[0,1)$, then $$|\nabla f(z)|\le 2\frac{1-f(z)^2}{1-|z|^2} \text{\ \,for all } z\in\mathbb{D},$$ so that
$$d_h(f(z), f(w))\le 2 d_h(z,w) \text{\ \,for all }  z,w\in\mathbb{D}.$$ Further, if $f(0)=0$, then $$|f(z)|\le \frac{4}{\pi}\tan^{-1} |z|  \text{\ \,for all } z\in\mathbb{D}.$$
\end{corollary}
\section{Proof of main results}
Theorem~\ref{main} below is the main part of the Theorem~\ref{prop}, and it solves Problem~\ref{prob} for the modified metrics.

\begin{theorem}\label{main}
Assume that $R$ is a metric of non-negative Gaussian curvature in $(-1,1)$. If $f$ is an $R$-harmonic function of the unit disk into $(-1,1)$, then it satisfies the sharp
inequalities 
\begin{equation}\label{shine}
|\nabla f(z)|\le \frac{4}{\pi}\frac{1-f(z)^2}{1-|z|^2}
\end{equation}
and  \begin{equation}\label{444}d_h(f(z), f(w))\le \frac{4}{\pi} d_h(z,w)\end{equation}
for $z,w\in\mathbb{D}$, where $d_h$ is the hyperbolic metric.
\end{theorem}
To prove Theorem~\ref{main}, we need the following lemma, which is of interest in its own right.  
\begin{lemma}\label{propi1}
Assume that $f$ is an increasing $\mathcal{C}^1$ diffeomorphism of $[-1,1]$ onto itself such that $f'$ is log-concave. Then for all $x\in[-1,1]$ we have the inequality
\begin{equation}\label{gine}1-f(x)^2\le f'(x)(1-x^2).\end{equation}
\end{lemma}
\begin{proof}[{Proof of Lemma~\ref{propi1}}]
Let $h:=\log (f')$. Take any $x\in(-1,1)$. Since $h$ is concave, for some real $k$ and all $t\in(-1,1)$ we have
\begin{equation}
	h(t)\le h_x(t):=h(x)+k(t-x);
\end{equation}
by approximation, without loss of generality $k\ne0$.
Also, the condition that $f$ is an increasing diffeomorphism of $[-1,1]$ onto itself implies that $f(-1)=-1$ and $f(1)=1$.
So,
\begin{equation*}
\begin{aligned}
	f(x)&=1-\int_x^1  f'(t)dt\, \\
	&=1-\int_x^1 e^{h(t)} dt\,  \\
	&\ge1-\int_x^1 e^{h_x(t)} dt\,  \\
&	=g^+(U,x):=1-U\frac{1-e^{k(1-x)}}{-k},
\end{aligned}
\end{equation*}
where
\begin{equation}
	U:=e^{h(x)}=f'(x)>0.
\end{equation}
Similarly,
\begin{equation*}
\begin{aligned}
	f(x)&=-1+\int_{-1}^x e^{h(t)} dt\,  \\
	&\le-1+\int_{-1}^x e^{h_x(t)} dt\,  \\
&	=-g^-(U,x):=-1+U\frac{e^{-k(1+x)}-1}{-k}.
\end{aligned}
\end{equation*}
So,
\begin{equation}
	f(x)^2\ge g_2(U,k,x):=\max[g^+(U,x)_+^2,g^-(U,x)_+^2],
\end{equation}
where $z_+:=\max(0,z)$.

We also have
\begin{equation*}
	2=\int_{-1}^1 f'=\int_{-1}^1 e^h\le \int_{-1}^1 e^{h_x(t)}\,dt
	=U e^{-k x}\,\frac{2\sinh k}k,
\end{equation*}
so that
\begin{equation}
	U\ge U_{k,x}:=e^{k x}\,\frac k{\sinh k}.
\end{equation}

Thus, it is enough to show that
\begin{equation}
	\rho(U,k,x):=\frac{1-g_2(U,k,x)}{U(1-x^2)}\le1
\end{equation}
for $U\ge U_{k,x}$.
Note that $\rho(U,k,x)(1-x^2)$ is a continuous piecewise-rational function of $U$ such that $\R$ can be partitioned into several intervals such that on each of the intervals of the partition the expression $\rho(U,k,x)(1-x^2)$ coincides with one of the following three expressions:
\begin{equation}
	\rho_+(U):=\frac{1-g^+(U,x)^2}U,\quad
	\rho_-(U):=\frac{1-g^-(U,x)^2}U,\quad
	\rho_0(U):=\frac{1-0}U.
\end{equation}
We have
\begin{equation}
	\rho'_+(U)=-\frac{\left(e^{k-k x}-1\right)^2}{k^2}\le0,\quad
		\rho'_-(U)=-\frac{e^{-2 k (x+1)} \left(e^{k (x+1)}-1\right)^2}{k^2}\le0,\quad
\end{equation}
and $\rho'_0(U)<0$.

So, $\rho(U,k,x)$ is nonincreasing in $U$. It remains to show that
\begin{equation}
	r(k,x):=\rho(U_{k,x},k,x)\le1.  \label{eq:r<1}
\end{equation}

Note that $g^+(U_{k,x},x)=-g^-(U_{k,x},x)=\left(e^{k x}-e^k\right) \csch k+1$. So,
\begin{equation}
	r(k,x)=\frac{1-g^+(U_{k,x},x)^2}{U_{k,x}(1-x^2)}
	=  \frac{2 (\cosh k-\cosh k x)\csch k}{k \left(1-x^2\right)}.
\end{equation}

Inequality \eqref{eq:r<1} can be rewritten as
\begin{equation}\label{eq:dif<0}
\dif(x):=
(1-x^2)r(k,x)-(1-x^2)\le0
\end{equation}
for real $k>0$ and $x\in[0,1]$, because $\dif(-k,x)=\dif(k,x)=\dif(k,-x)$.

We have $\dif'''(x)=-2 k^2 \csch k\, \sinh k x\le0$. So, $\dif''$ is nonincreasing. Hence, there is some $c\in[0,1]$ such that $\dif$ is convex on $[0,c]$ and concave on $[c,1]$. Also, $\dif'(0)=\dif'(1)=\dif(1)=0$. Thus, \eqref{eq:dif<0} follows.
\end{proof}

\begin{proof}[Proof of Theorem~\ref{main}] Let us show that  $R\in\mathcal{L}^1(-1,1).$
In view of \eqref{gkrr}, $\log R$ is concave. Therefore,
$$\log R(t)\le \log R(0)+\frac{R'(0)}{R(0)} t \text{\ \,for all } t\in(-1,1),$$ and thus $$R(t)\le R(0) e^{\frac{R'(0)}{R(0)} t}.$$
Hence, $$\int_{-1}^1 R(t)dt<\infty.$$
Now we put  $$r:=H(1)  = \frac{1}{2}\int_{-1}^1 R(u)du.$$
Recall the Euclidean harmonic function $g$ defined in \eqref{gg}.  It comes down to estimating the gradient of the derivative of the function $ g $, which is equal to $$|\nabla g|=R(f) |\nabla f|/H(1).$$
For the real Euclidean harmonic function $g:\mathbb{D}\to (-1,1)$ we have \cite{chenm, kavu}:
\begin{equation}\label{3main}|\nabla g(z)|=\frac{R(f(z)) |\nabla f(z)|}{r}\le \frac{4}{\pi}\frac{\cos \frac{\pi}{2}g(z)}{1-|z|^2},\end{equation}
where $$g(z):=\frac{1}{r}\left(H(0)+ \int_0^{f(z)}R(u)du\right).$$
Let \begin{equation}\label{rrr}\mathcal{R}:=\frac{4}{\pi}\,\frac{\cos\left(\frac{\pi}{2r}\left(H(0)+ \int_0^{f(z)}R(u)du\right)\right)}{1-|z|^2}.\end{equation}
Note that $$\cos\frac{\pi}{2} b\le 
1-b^2$$ for $b\in[0,1]$.

Let $$
\psi(v):=\frac{\pi}{2r}\left(H(0)+ \int_0^{v}R(u)du\right)$$ and apply Lemma~\ref{propi1}. We get
\begin{equation}\label{14may}\mathcal{R}\le
\frac4\pi\,\frac{1-\psi(v)^2}{\psi'(v)}
\le\frac4\pi\,(1-v^2).\end{equation}

Combining \eqref{rrr}, \eqref{3main} and \eqref{14may}, we obtain \eqref{shine}. Concerning  \eqref{444}, notice that the proof of \cite[Theorem~1.2]{kavu} can be applied in this case, because the \break $\rho$-harmonicity is invariant under precomposition by M\"obius transformations.
\end{proof}
The following example shows that one cannot omit the condition of positive Gaussian
curvature. In fact, we cannot prove a weaker estimate with a constant larger than $4/\pi
$.
\begin{example}\label{myex}
Assume that $\varrho(w):=\frac{1}{1-|w|^2}$ and let $\rho(w):=\varrho(u,0)=\frac{1}{1-u^2}$. Then $f: \mathbb{D}\to (-1,1)$ is $\rho$-harmonic  (and $\varrho$-harmonic) if and only if $$f(z) = \tanh g(z)$$ for a Euclidean harmonic mapping $g$ of the unit disk in the real line $\mathbb{R}$. In particular, the functions $$f(z) =\tanh (nx), \ \ z=x+\imath y, \ \ n\in\Bbb{N},$$ are $\rho$-harmonic. Then $|\nabla f(z)|=n \mathrm{sech}\,^2(n x),$ and so, $$\frac{|\nabla f(z)|}{1-|f(z)|^2}(1-|z|^2)\bigg|_{z=0}=n,$$ so that in Theorem~\ref{main} we cannot omit the condition of the positiveness of Gaussian curvature, nor can we even prove a weaker statement with a larger constant factor instead of $4/\pi$. Observe that in this case $$\mathcal{K}_\rho(z) =-2(1+x^2)<0. $$ Of course, the curvature of the hyperbolic (Poincar$\grave{\mathrm{e}}$) metric is $\mathcal{K}_\varrho(z)=-4$, and
it is not equal to the curvature of $\rho$, 
even though both curvatures are negative.
\end{example}

In the following example it is given a result for a metric $R$ of zero Gaussian curvature.
\begin{example}
Assume that the Gaussian curvature of $R$ is zero. Then  $R(x)=e^{cx}$.  Moreover, by \eqref{3main},
$$|\nabla f(z)|\le A:=\frac{4 e^{-c f(z)} \sin\left[\frac{\pi}{2 \sinh(c)} \left(e^c-e^{c f(z)}\right)    \right] \sinh(c)}{c  \pi  \left(1-|z|^2\right)}.$$
Further, by the proof of Theorem~\ref{main},  $$|\nabla f(z)|\le A\le  \frac{4}{\pi}\frac{1-f(z)^2}{1-|z|^2}.$$
\end{example}



\subsection{Proof of Theorem~\ref{kalajpos}}\label{jun2022}
We need the following lemma.
\begin{lemma}\label{lema}  If $R\colon(-1,1)\to (0,+\infty)$ is positive, increasing in $(-1,0)$ and decreasing in $(0,1)$, if $v\in(-1,1)$, and if  
$$r=\frac{1}{2}\int_{-1}^1 R(u) \, du,$$
then we have the sharp inequality
\begin{equation}\label{leqleq}\sin\left[\frac{\pi  \int_v^1 R(u) \, du}{2r}\right]\le \frac{\pi}{2r}{\left(1-|v|\right) R(v)},\end{equation} and in particular
$$\sin\left[\frac{\pi  \int_v^1 R(u) \, du}{2 r}\right]\le \frac{\pi}{2r} {\left(1-v^2\right) R(v)}.$$

The constant $\pi/2$ is sharp even if we restrict the consideration to $\mathcal{C}^2$ diffeomorphisms $R\colon(-1,1)\to (0,\infty)$.



\end{lemma}

\begin{proof}[Proof of Lemma~\ref{lema}]
The proof of inequality \eqref{leqleq} is easy. We use here the fact that $R$ is decreasing in $[0,1)$ and the elementary inequality $\sin x\le x$ for $x\in[0,\pi/2]$. Then for $v\in[0,1]$ we have
\[
\begin{split}
\sin\left[\frac{\pi  \int_v^1 R(u) \, du}{ 2r}\right]&\le \frac{\pi  \int_v^1 R(u) \, du}{ 2r}\\
 &\le \pi \frac{\left(1-v\right) R(v)}{2r}.
\end{split}
\]
If $v<0$, then we use the fact that $R$ is increaing in $(-1,0)$. We come to the desired inequality as follows $$\sin\left(\frac{\pi}{2r} \int_v^{1}R(u)du\right)=\sin\left(\frac{\pi}{2r} \int_{-1}^{v}R(u)du\right)\le \pi \frac{\left(1+v\right) R(v)}{2r}.$$
To prove the sharpness part, observe that inequality \eqref{leqleq} is equivalent to \begin{equation}\label{itfoll}\cos \phi(v)\le(1-v^2)\phi'(v), \end{equation}
where
$$ \phi(v)=\frac{\pi}{2}-\frac{\pi  \int_v^1 R(u) \, du}{2 \int_0^1 R(u) \, du}=\frac{\pi  \int_0^v R(u) \, du}{2 \int_0^1 R(u) \, du}.$$
For $s, as^2\in(0,1)$ we define the concave diffeomorphism $\psi\colon[0,1]\to[0,1]$ by the formula
\begin{equation}\label{foraf}\psi(x):= \left\{
\begin{array}{ll}
 \left(1+2 a s-a s^2\right) x-a x^2 & \text{if }x>0\wedge x<s, \\
 1+\left(1-a s^2\right) (-1+x) & \text{if }x\geq s\wedge x\leq 1.
\end{array}
\right.\end{equation}
Now we define $\phi(x) = \frac{\pi}{2}\psi(x)$.

Then for $u=as^2$ we have $$\frac{\cos \phi(s)}{(1-s^2)\phi'(s)}=\frac{2\sin\left[\frac{1}{2} \pi  (1-s) (1-u)\right]}{\pi \left(1-s^2\right) (1-u)}.$$ The supremum of the latter expression is equal to $1$. It is ``attained in the limit'', for instance, if $s=1/n\to 0$ and $u = (n-1)^2/n^2=as^2\to 1$, with $n\to \infty$.

To prove the last statement, extend $\psi$ in $[-1,1]$ by $\psi(x)= -\psi(-x)$ and define $R(x)=\psi'(x)$, for $x\in[-1,1]$. Then $R$ is not smooth, but it is continuous on $[-1,1]$. 

We introduce appropriate mollifiers: Fix a smooth even function
$\sigma:\R\to[0,1]$ which is compactly supported in the interval
$(-1,1)$ and satisfies $\int_\R\sigma=1$. For $\varepsilon>0$ consider the
mollifier
\begin{equation}\label{mol}
     \sigma_\varepsilon(t):=\frac{1}{\varepsilon}\,
     \sigma\left(\frac{t}{\varepsilon}\right).
\end{equation}
It is compactly supported in the interval $(-\varepsilon,\varepsilon)$ and
satisfies $\int_\R\sigma_\varepsilon=1$. For $\varepsilon>0$ define
$$\varphi_\varepsilon(x): = \int_{\Bbb R} \psi(y)
\frac{1}{\varepsilon}\sigma(\frac{x-y}{\varepsilon})dy=\int_{\Bbb R}
\psi(x-\varepsilon z)\sigma(z)dz.$$
Because $\sigma$ is even we have
\[\begin{split}\varphi_\varepsilon(-x) &= \int_{\Bbb R}
\psi(-x-\varepsilon z)\sigma(z)dz\\&=\int_{\Bbb R}
\psi(x+\varepsilon z)\sigma(z)dz\\&=-\int_{\Bbb R}
\psi(x-\varepsilon z)\sigma(z)dz\\&=-\varphi_\varepsilon(x),\end{split}\]
and
 $$\varphi'_\varepsilon(x) =\int_{\Bbb R} \psi'(x-\varepsilon
z)\sigma(z)dz.$$ So $\varphi_\varepsilon$ is an increasing and odd function. Further we define $\psi_\varepsilon(x):=\frac{1}{\varphi_\varepsilon(1)}\varphi_\varepsilon(x)$. Then $\psi_\varepsilon(x):[-1,1]\to [-1,1]$ is a $\mathcal{C}^\infty$ increasing odd diffeomorphism. Then $\psi_\varepsilon(x)$ converges uniformly to $\psi$ and $\psi'_\varepsilon(x)$ converges uniformly to $\psi'(x)$ as $\varepsilon\to 0$. Thus the function $R_\varepsilon(x) =\psi'_\varepsilon(x)$ is an even function, increasing in $[-1,0]$ and decreasing in $[0,1]$ that   converges uniformly to $R$.

This implies that the constant $\pi/2$ is sharp even if we restrict the consideration to $\mathcal{C}^\infty$ diffeomorphisms.
\end{proof}

\begin{proof}[Proof of Theorem~\ref{kalajpos}] Since $R$ is positive, increasing in $(-1,0)$ and decreasing in $(0,1)$, it is clear that  $R\in\mathcal{L}(-1,1)$.
From \eqref{3main} we have \begin{equation}|\nabla g(z)|=\frac{R(f(z)) |\nabla f(z)|}{r}\le \frac{4}{\pi}\frac{\cos \frac{\pi}{2}g(z)}{1-|z|^2},\end{equation}
where $$g(z):=\frac{H(0)}{2r}+\frac{1}{r}\left( \int_0^{f(z)}R(u)du\right),$$ with $H(0)$ defined in \eqref{nzero}.

Further, $$\cos \frac{\pi}{2}g(z)=\sin\left[\frac{\pi  \int_v^1 R(u) \, du}{2r}\right].$$

In view Lemma~\ref{lema}, the inequality \eqref{222} is proved.
 To prove \eqref{555}, in view of the assumption, we observe first that $H(0)=0$. Since the function $\psi(u) =\int_0^u R(t)\frac{dt}{r}$ is concave on $[0,1]$ with $\psi(0)=\psi(1)-1=0$,  it satisfies the inequality $\psi(u)\ge u$. Therefore $$|f(z)|\le \frac{1}{r}\left| \int_0^{f(z)}R(u)du\right|=|g(z)|.$$
Now we use the Schwarz lemma for Euclidean harmonic functions (\cite{he, Duren2004}), which implies that $$|g(z)|\le \frac{4}{\pi}\tan^{-1} |z|.$$

Inequality \eqref{222} is sharp because of Lemma~\ref{lema}.
Inequality \eqref{555} is sharp, since it  coincides with the corresponding inequality \cite[p.~124]{ABR} for Euclidean harmonic mappings (planar case), where the sharpness part is established. Observe that, if $R\equiv 1$, then $R$ defines the Euclidean metric and satisfies the conditions of our theorem.  This finishes the proof of the theorem.
\end{proof}

\section{Concluding remarks}\label{subse}
The answer to the general question posed in  Problem~\ref{prob} is negative. In the following example it is shown that for metrics of zero Gaussian curvature, the quantity $S(f)$ defined in \eqref{sf}, can be arbitrary big.

\begin{example}\label{primjer}  
For $z=x+\imath y$, let $g(z)=\imath k y$, where $k>0$,  
and  assume that $\phi$ is a conformal automorphism of
$\mathcal{S}=\{x+\imath y: x\in(-1,1), y\in\mathbb{R}\}$  which maps $y$-axis onto $(-1,1)$. Let  $g_1=\phi \circ g$. 
For instance, one may define a conformal automorphism $\phi$ as follows:
$$\phi(z) := -\frac{2 \imath \log\left[-\imath+\frac{2}{-\imath+e^{\frac{\imath \pi  z}{2}}}\right]}{\pi }.$$ Next let
 $\varrho(w)=|\zeta'(w)|$, where we use notation $w=\phi(\zeta)$   and
$w\mapsto \zeta(w)$ denotes the inverse function to $\phi$. Then $g_1$
is   $\varrho$-harmonic, $\lambda_0(\imath y)=\pi/2$  and $|\nabla g_1(\imath y)|=k|\phi'(\imath y)$. Also, $|\nabla g_1(0)|=2k$.   Here $\lambda_0(z)$ is the hyperbolic metric of the strip.
Since the expression \eqref{sf} is invariant with respect to conformal maps and hyperbolic
metrics, by taking a conformal mapping $a$ of the unit disk onto the strip $\mathcal{S}$ satisfying $a(0)=0$, and defining $f(z)=g_1(a(z))$, we see that
$$S(f(z))=\frac{|\nabla f(z)|(1-|z|^2)}{1-|f(z)|^2}=\frac{|\nabla g_1(a(z))|}{(1-g_1(a(z))^2)\lambda_0(a(z))}$$
can be arbitrary big for $z=0$, namely $S(f(0))=4k/\pi$. We remark that in this case $$\varrho^2 (u,v)=\frac{2}{\cos [\pi  u]+\cosh[\pi  v]},$$ so that $\mathcal{K}_\varrho=0,$ but $\mathcal{K}_\rho=- \pi ^2 /4$, where $\rho(u,v)=\varrho(u,0)$.
\end{example}

The following example raises a similar question for positive harmonic functions.

\begin{example} It is  well known that a positive harmonic function defined in the half-plane is a contraction with respect to hyperbolic metric (see e.g. \cite{mmarkovic}). So, it is natural to ask whether such a result is true for positive $R$-harmonic functions defined in the half-plane, where $R$ is a metric of non-negative Gaussian curvature. The following example shows that this is not true. Let $R(x)=1-e^{-x}$ and define the positive $R$-harmonic function on $S(0,\infty):=\{x+\imath y: x>0, y\in\mathbb{R}\}$ by $$f(x,y) := \log\left[\frac{\pi }{\frac{\pi }{2}-\tan^{-1}\left[\frac{y}{x}\right]}\right]=R\left(\Re \left[-\imath/\pi \log(\imath z)\right]\right).$$ Observe that $-\log (R(x))''=\frac{1}{4} \csch\left[\frac{x}{2}\right]^2$. So, $R$ has a non-negative curvature.

On the other hand,
$$x\frac{|\nabla f(x,y)|}{{f(x,y)}}=\frac{2 x \sqrt{\frac{1}{\left(x^2+y^2\right) \left(\pi -2 \tan^{-1}\left[\frac{y}{x}\right]\right)^2}}}{\log\left[\frac{2 \pi }{\pi -2 \tan^{-1}\left[\frac{y}{x}\right]}\right]}=\frac{2 \sqrt{\frac{1}{\left(1+t^2\right) (\pi -2 \tan^{-1}[t])^2}}}{\log\left[\frac{2 \pi }{\pi -2 \tan^{-1}[t]}\right]}$$ for $y=t x$. The last expression has its maximum at $t=-1.4771\dots$ and it is equal to $1.0482\dots$. This implies, in particular, that $f\colon S(0,\infty)\to (0,+\infty)$ is not a contraction with respect to corresponding hyperbolic metrics. It would be of interest to find the best Lipschitz constant in this context.
\end{example}

\subsection*{Competing interests} The authors declare none.
\subsection*{Acknowledgments} We would like to thank the anonymous referee for very helpful
comments that had a significant impact on this paper.

\end{document}